\documentclass[11pt,a4paper]{amsart}
\usepackage[utf8]{inputenc}
\usepackage[english]{babel}
\usepackage[T1]{fontenc}
\usepackage{lmodern} 
\rmfamily 
\DeclareFontShape{T1}{lmr}{b}{sc}{<->ssub*cmr/bx/sc}{}
\DeclareFontShape{T1}{lmr}{bx}{sc}{<->ssub*cmr/bx/sc}{}

\usepackage{sidecap}

\usepackage{verbatim}
\usepackage{lmodern}
\usepackage{amsmath}
\usepackage{amssymb} 
\usepackage{amsthm} 
\usepackage{thmtools}
\usepackage{enumitem}
\usepackage{graphicx}
\usepackage[hidelinks]{hyperref}
\usepackage{mathtools}
\usepackage{indentfirst}
\usepackage{tabularx}
\usepackage{bbm}
\usepackage[capitalise]{cleveref}
\usepackage{stmaryrd}
\usepackage{tabularx}

\usepackage[left=1.2in,top=2cm,right=1.2in,bottom=2cm]{geometry}

\usepackage[
textwidth=2cm,
textsize=small,
colorinlistoftodos]
{todonotes} 

\usepackage{tikz} 
\usepackage{tikz-cd} 
\newcommand{\btk}{\begin{tikzcd}}
\newcommand{\etk}{\end{tikzcd}}
\usetikzlibrary{arrows,calc,decorations.markings}
\usepackage{xparse}

\usepackage{blindtext}

\makeatletter
\newcommand{\@bbify}[1]{
  \ifcsname b#1\endcsname
  \message{WARNING: Overwriting b#1 with blackboard letter!}
  \fi
  \expandafter\edef\csname b#1\endcsname
  {\noexpand\ensuremath{\noexpand\mathbb #1}\noexpand\xspace}}
\newcommand{\@calify}[1]{
  \ifcsname c#1\endcsname
  \message{WARNING: Overwriting c#1 with calligraphic letter!}
  \fi 
  \expandafter\edef\csname c#1\endcsname
  {\noexpand\ensuremath{\noexpand\mathcal #1}\noexpand\xspace}}
\newcommand{\@bfify}[1]{
  \ifcsname bf#1\endcsname
  \message{WARNING: Overwriting c#1 with bold letter!}
  \fi
  \expandafter\edef\csname bf#1\endcsname
  {\noexpand\ensuremath{\noexpand\mathbf #1}\noexpand\xspace}}
\newcounter{@letter}\stepcounter{@letter}
\loop\@bbify{\Alph{@letter}}\@calify{\Alph{@letter}}\@bfify{\Alph{@letter}}
\ifnum\the@letter<26\stepcounter{@letter}\repeat
\makeatother

\newenvironment{tz}{\begin{center}\begin{tikzpicture}}{\end{tikzpicture}\end{center}}
\tikzstyle{d}=[double distance=.3ex]
\tikzstyle{w}=[preaction={draw=white,-,line width=5pt}]

\tikzset{%
node distance=1.5cm, la/.style={scale=0.8}, lasmall/.style={scale=0.75}, over/.style={auto=false,fill=white,inner sep=1.5pt, minimum size=0, outer sep=0},
    symbol/.style={%
        draw=none,
        every to/.append style={%
            edge node={node [sloped, allow upside down, auto=false]{$#1$}}},
            
    }, pro/.style={postaction={decorate,decoration={
        markings,
        mark=at position .5 with {\node at (0,0) {$\bullet$};}
      }},
      inner sep=.9ex,
      },
      prosmall/.style={postaction={decorate,decoration={
        markings,
        mark=at position .5 with {\node at (0,0) {$\scriptstyle \bullet$};}
      }},
      inner sep=.9ex,
      },
  n/.style={double equal sign distance, -implies}, t/.style={double distance=2.5pt, -implies, postaction={draw,-}},
}

\newcommand{\Cof}{\mathrm{Cof}}

\newcommand{\cof}{\ensuremath{\mathrm{cof}}}

\newcommand{\fib}{\ensuremath{\mathrm{Fib}}}

\DeclareMathOperator{\Sing}{Sing}
\DeclareMathOperator{\Map}{Map}

\newcommand{\Path}{\mathrm{Path}}

\newcommand{\C}{\mathcal{C}}

\newcommand{\J}{\mathcal{J}}
\newcommand{\K}{\mathcal{K}}

\newcommand{\W}{\mathcal{W}}

\newcommand{\symsp}{\mathrm{Sp}_{\mathrm{st}}^\Sigma}
\newcommand{\seqsp}{\mathrm{Sp}_{\mathrm{st}}^{\mathbb{N}}}

\newlist{rome}{enumerate}{7}
\setlist[rome]{label=(\roman*)}

\newtheorem{theorem}{Theorem}[section]

\theoremstyle{definition}

\theoremstyle{remark}
\newtheorem{rem}[theorem]{Remark}

\crefname{theorem}{Theorem}{Theorems}
\crefname{cor}{Corollary}{Corollaries}
\crefname{prop}{Proposition}{Propositions}
\crefname{lem}{Lemma}{Lemmas}

\crefname{defn}{Definition}{Definitions}
\crefname{terminology}{Terminology}{Terminologies}
\crefname{ex}{Example}{Examples}
\crefname{notation}{Notation}{Notations}
\crefname{descr}{Description}{Descriptions}
\crefname{constr}{Construction}{Constructions}

\crefname{rem}{Remark}{Remarks}

\usepackage{blindtext}

\renewcommand\thepart{\Roman{part}.}

\makeatletter
\renewcommand\part{%
  \par
  \addvspace{4ex}%
  \@afterindenttrue
  \secdef\@part\@spart
}

\def\@part[#1]#2{%
    \ifnum \c@secnumdepth >\m@ne
      \refstepcounter{part}%
      
      \addcontentsline{toc}{section}{\hspace{-.5cm} \bfseries\thepart\hspace{1em}#1}%
    \else
      \addcontentsline{toc}{section}{#1}%
    \fi
    {\parindent \z@ \raggedright
     \interlinepenalty \@M
     \normalfont
     \thispagestyle{empty}
     \ifnum \c@secnumdepth >\m@ne
      \centering\large\textsc{\textbf{\thepart}}\nobreakspace
     \fi
     \centering\large\textsc{\textbf{#2}}
     \par}%
    \nobreak
    \vskip .3cm
    \@afterheading}
\def\@spart#1{%
      \addcontentsline{toc}{part}{#1}%
    {\parindent \z@ \raggedright
     \interlinepenalty \@M
     \normalfont
     \thispagestyle{plain}
     \centering\large\textsc{\textbf{#1}}
     \par}%
    \nobreak
    \vskip .3cm
    \@afterheading}
\makeatother

\title{A concise proof of the stable model structure on symmetric spectra}

\author[C.\ Malkiewich]{Cary Malkiewich}
\address{Department of Mathematics,
Binghamton University,
Binghamton, NY 13902, USA}
\email{malkiewich@math.binghamton.edu}

\author[M.\ Sarazola]{Maru Sarazola}
\address{School of Mathematics, University of Minnesota, Minneapolis MN, 55415, USA}
\email{maru@umn.edu}

\begin{document}

\maketitle

\begin{abstract}
It is well-known that the stable model structure on symmetric spectra cannot be transferred from the one on sequential spectra through the forgetful functor. We use the fibrant transfer theorem of \cite{fibrantMS} to show it can be transferred between fibrant objects, providing a new, short and conceptual proof of its existence.
\end{abstract}

\section{Introduction}

Symmetric spectra provide one of the most general and flexible frameworks for stable homotopy theory, as a symmetric monoidal model category. There are other model categories with good smash products, including orthogonal spectra from \cite{MMSS} and the $S$-modules of \cite{EKMM}, but symmetric spectra are more general than these, because they only use symmetric groups in their definition. This allows us to define them in settings where a topological enrichment is not naturally present, such as the motivic setting \cite{hovey_general}.

However, the theory of symmetric spectra has a famous drawback: the stable equivalences cannot be defined as the maps that induce isomorphisms on the homotopy groups (the $\pi_*$-isomorphisms), at least if we give the homotopy groups their most obvious definition. As a result, the proof of the stable model structure for symmetric spectra in \cite{HSS} and \cite{MMSS} requires additional concepts and intermediate results beyond those needed to establish the same model structure for sequential spectra (prespectra) or orthogonal spectra. This makes the theory somewhat less ``user friendly'' and presents a significant tradeoff compared to the other models.

A natural attempt to avoid these complications is to use the transfer theorem (see for instance \cite[Section 3]{crans} or the dual of \cite[Theorem 2.2.1]{HKRS}) to produce the model structure from the adjunction relating symmetric spectra to sequential spectra,
\begin{tz}
\node[](1) {$\mathrm{Sp}^\Sigma$}; 
\node[right of=1,xshift=1cm](2) {$\seqsp$.}; 

\draw[->] ($(1.east)-(0,5pt)$) to node[below,la]{$U$} ($(2.west)-(0,5pt)$);
\draw[->] ($(2.west)+(0,5pt)$) to node[above,la]{$F$} ($(1.east)+(0,5pt)$);

\node[la] at ($(1.east)!0.5!(2.west)$) {$\bot$};
\end{tz}
If successful, this would construct a model structure on $\mathrm{Sp}^\Sigma$ in which a map $f$ is a weak equivalence if and only if $Uf$ is a $\pi_*$-isomorphism in $\seqsp$. Sadly, this approach cannot succeed, because it produces the wrong class of weak equivalences: we would get the maps inducing isomorphisms on homotopy groups, rather than the maps inducing isomorphisms on generalized cohomology. These weak equivalences are ``wrong'' because with them the adjunction $(F,U)$ is not a Quillen equivalence, and so we get the wrong homotopy category on the left-hand side.

Upon inspection, we see that the discrepancy between the classes of weak equivalences vanishes once we restrict to maps between fibrant objects. This suggests that one could use a more flexible version of the transfer theorem introduced recently in \cite{fibrantMS}, where weak equivalences are transferred through the right adjoint only between fibrant objects, leaving room for equivalences between arbitrary objects to admit a more complicated description. 

The goal of this note is to prove that this approach works, providing a short proof of the existence of the stable model structure on $\symsp$. Notably, it allows us to completely avoid a treatment---and even a definition of---stable equivalences between non-fibrant objects, and only rely on well-known facts about $\pi_*$-isomorphisms in $\seqsp$.

\section{Review of model structures on spectra}

In what follows, we let $\mathrm{Sp}^\mathbb{N}$ denote the category of sequential spectra based on simplicial sets. These are sequences of pointed simplicial sets $X_n$ with bonding maps $X_n \wedge S^1 \to X_n$, where $S^1$ denotes the simplicial circle. We similarly let $\mathrm{Sp}^\Sigma$ denote the corresponding category of symmetric spectra, in which the levels $X_n$ have actions by the symmetric group, and the bonding maps respect these actions in the sense of \cite[Def 1.2.2]{HSS}.

We briefly recall some of the model structures that can be defined on these categories. The first two are ``level'' model structures, in which the weak equivalences are maps that are equivalences at each spectrum level separately.

\begin{theorem}\cite[Proposition 2.2]{BF}
There is a \emph{level model structure} on sequential spectra $\mathrm{Sp}_l^\mathbb{N}$, in which weak equivalences, fibrations, and trivial fibrations are defined as the levelwise corresponding maps on $\mathrm{sSet}_{*,\mathrm{Quillen}}$.
\end{theorem}

\begin{theorem}\cite[Theorem 5.1.2]{HSS}\label{thm:stableMSsymsp}
There is a \emph{level model structure} on symmetric spectra $\mathrm{Sp}_l^\Sigma$, in which weak equivalences, fibrations, and trivial fibrations are defined as the levelwise corresponding maps on $\mathrm{sSet}_{*,\mathrm{Quillen}}$.
\end{theorem}

Of course, the levelwise equivalences are not the equivalences we want in stable homotopy theory---we want a broader class of maps called stable equivalences. For sequential spectra, these are the same thing as $\pi_*$-isomorphisms.

\begin{theorem}\cite[Proposition 2.3]{BF}
There is a \emph{stable model structure} on sequential spectra $\seqsp$, in which:
\begin{itemize}
\item weak equivalences are the $\pi_*$-isomorphisms; i.e. the maps inducing isomorphisms on all homotopy groups,
\item fibrant objects are the $\Omega$-spectra ($X_n \to \Omega X_{n+1}$ is an equivalence) which are levelwise Kan complexes, and
\item trivial fibrations are the level trivial fibrations.
\end{itemize}
\end{theorem}

\begin{theorem}\cite[Theorem 3.4.4]{HSS}\label{thm:stableMSsymsp}
There is a \emph{stable model structure} on symmetric spectra $\symsp$, in which:
\begin{itemize}
\item weak equivalences are the stable equivalences; i.e. the maps $f\colon X\to Y$ such that $f^*\colon \Map(X,E)\to\Map(Y,E)$ is a weak equivalence of simplicial sets for every injective $\Omega$-spectrum $E$, and
\item fibrant objects and trivial fibrations agree with those of $\seqsp$.
\end{itemize}
\end{theorem}

All of these model structures are cofibrantly generated.

\begin{rem}
The definition of the stable equivalences for symmetric spectra requires the notion of an \emph{injective} spectrum (\cite[Definition 3.1.1]{HSS}), which is not recalled here. This is an intentional omission: the injectivity condition can be dropped when working between fibrant objects, which is all that is needed in our proof.
\end{rem}

\begin{rem}\label{rem:top}
Once the stable model structures are constructed, one can use the adjunction between spectra of topological spaces and spectra of simplicial sets
\begin{tz}
\node[](1) {$\mathrm{Sp}_\mathrm{Top}$}; 
\node[right of=1,xshift=1cm](2) {$\mathrm{Sp}.$}; 

\draw[->] ($(1.east)-(0,5pt)$) to node[below,la]{$\Sing$} ($(2.west)-(0,5pt)$);
\draw[->] ($(2.west)+(0,5pt)$) to node[above,la]{$\vert - \vert$} ($(1.east)+(0,5pt)$);

\node[la] at ($(1.east)!0.5!(2.west)$) {$\bot$};
\end{tz}
By the classical transfer theorem, we get stable model structures on the categories of spectra based on topological spaces, $\mathrm{Sp}_\mathrm{Top}^\mathbb{N}$ and $\mathrm{Sp}_\mathrm{Top}^\Sigma$. As expected, each of these $\mathrm{Top}$-based model structures is Quillen equivalent to its simplicial counterpart.
\end{rem}

Sequential and symmetric spectra are related by a free-forgetful adjunction \cite[Thm 4.3.2]{HSS}
\begin{tz}
\node[](1) {$\mathrm{Sp}^\Sigma$}; 
\node[right of=1,xshift=1cm](2) {$\mathrm{Sp}^\mathbb{N}$}; 

\draw[->] ($(1.east)-(0,5pt)$) to node[below,la]{$U$} ($(2.west)-(0,5pt)$);
\draw[->] ($(2.west)+(0,5pt)$) to node[above,la]{$F$} ($(1.east)+(0,5pt)$);

\node[la] at ($(1.east)!0.5!(2.west)$) {$\bot$};
\end{tz}
which can be used to transfer the level model structure on sequential spectra to obtain the level model structure on symmetric spectra. A natural question is then: is it possible to transfer the stable model structure through the free-forgetful adjunction as well? Unfortunately this strategy is not fruitful.

\begin{rem}\label{remark}
As observed by Hovey--Shipley--Smith in \cite[Section 3.1]{HSS}, the model structure on $\symsp$ cannot be transferred through the free-forgetful adjunction \begin{tz}
\node[](1) {$\mathrm{Sp}^\Sigma$}; 
\node[right of=1,xshift=1cm](2) {$\seqsp$}; 

\draw[->] ($(1.east)-(0,5pt)$) to node[below,la]{$U$} ($(2.west)-(0,5pt)$);
\draw[->] ($(2.west)+(0,5pt)$) to node[above,la]{$F$} ($(1.east)+(0,5pt)$);

\node[la] at ($(1.east)!0.5!(2.west)$) {$\bot$};
\end{tz}
Indeed, the functor $U$ does not preserve weak equivalences, as illustrated in \cite[Example 3.1.10]{HSS}. However, if we work only between fibrant objects, we find that a map $f$ is a stable equivalence if and only if $Uf$ is \cite[Corollary 4.3.4]{HSS}.
\end{rem}

\section{Existence of the model structure}

\cref{remark} suggests that the model structure on $\symsp$ can be constructed using the notion of fibrant transfer introduced in \cite{fibrantMS}. In this section we recall the statement of this fibrant transfer theorem, then use it to construct the model structure on $\symsp$.

\begin{theorem}\cite[Theorem 3.5]{fibrantMS}\label{thm:transfer}
Let $(\cM,\Cof,\fib,\W)$ be a combinatorial model category with generating set of trivial cofibrations $\K$, and let $\C$ be a locally presentable category. Suppose that we have an adjunction 
\begin{tz}
\node[](1) {$\C$}; 
\node[right of=1,xshift=1cm](2) {$\cM$}; 

\draw[->] ($(1.east)-(0,5pt)$) to node[below,la]{$U$} ($(2.west)-(0,5pt)$);
\draw[->] ($(2.west)+(0,5pt)$) to node[above,la]{$F$} ($(1.east)+(0,5pt)$);

\node[la] at ($(1.east)!0.5!(2.west)$) {$\bot$};
\end{tz}
and that the following properties are satisfied:
\begin{enumerate}[label=(\arabic*)]
    \item\label{srind2} for every map $f\colon X\to Y$ such that $Uf\in\fib\cap\W$, there is a commutative square
\begin{tz}
\node[](1) {$X$}; 
\node[right of=1](2) {$Y$}; 
\node[below of=1](1') {$X'$}; 
\node[below of=2](2') {$Y'$}; 

\draw[->] (1) to node[above,la]{$f$} (2); 
\draw[->] (1) to node[left,la]{$\iota_X$} (1'); 
\draw[->] (2) to node[right,la]{$\iota_Y$} (2'); 
\draw[->] (1') to node[below,la]{$f'$} (2');
\end{tz}
 with $UX',UY'$ fibrant, $\iota_X,\iota_Y\in \cof(F\K)$, and such that $Uf'\in\W$; 
    \item\label{srind1} for every object $X$ such that $UX$ is fibrant, there is a factorization of the diagonal morphism 
    \[ X\xrightarrow{w} \Path X\xrightarrow{p} X\times X \]
    with $Uw\in \W$ and $Up\in\fib$.
\end{enumerate}

Then, there exists a combinatorial model structure on $\C$ in which a morphism $f$ is a trivial fibration if and only if $Uf$ is so in $\cM$, and an object $X$ is fibrant if and only if $UX$ is fibrant in $\cM$. Moreover, a morphism $f$ between fibrant objects in $\C$ is a weak equivalence (resp.\ fibration) if and only if $Uf$ is so in $\cM$.
\end{theorem}

\begin{theorem}
The stable model structure on $\mathrm{Sp}^\Sigma$ can be fibrantly transferred through the free-forgetful adjunction
\begin{tz}
\node[](1) {$\mathrm{Sp}^\Sigma$}; 
\node[right of=1,xshift=1cm](2) {$\seqsp.$}; 

\draw[->] ($(1.east)-(0,5pt)$) to node[below,la]{$U$} ($(2.west)-(0,5pt)$);
\draw[->] ($(2.west)+(0,5pt)$) to node[above,la]{$F$} ($(1.east)+(0,5pt)$);

\node[la] at ($(1.east)!0.5!(2.west)$) {$\bot$};
\end{tz}
\end{theorem}
\begin{proof}
The categories of sequential and symmetric spectra are locally presentable, as we defined them from the locally presentable category of pointed simplicial sets. Moreover, the stable model structure $\seqsp$ is cofibrantly generated; let $\K$ denote the set of generating trivial cofibrations. We verify that conditions (1) and (2) of \cref{thm:transfer} hold. 

For condition (1), let $f\colon X\to Y$ be a map of symmetric spectra such that $Uf$ is a level trivial fibration. If $\fib$ denotes the class of fibrations in $\seqsp$, then the adjunction $F\dashv U$ gives a weak factorization system $$(\cof(F\K), U^{-1}(\fib)).$$ We can then consider a factorization of the map to the terminal spectrum
$$X\xrightarrow{\iota_X} X'\xrightarrow{p} 1$$ with $\iota_X\in\cof(F\K)$ and $Up\in\fib$; in particular, $UX'$ is fibrant. 

Taking the pushout of $\iota_X$ along $f$
\begin{tz}
\node[](1) {$X$}; 
\node[right of=1](2) {$Y$}; 
\node[below of=1](1') {$X'$}; 
\node[below of=2](2') {$X'\cup_X Y$}; 

\draw[->] (1) to node[above,la]{$f$} (2); 
\draw[->] (1) to node[left,la]{$\iota_X$} (1'); 
\draw[->] (2) to node[right,la]{$j$} (2'); 
\draw[->] (1') to node[below,la]{$g$} (2');
\end{tz}
produces a map $j\in\cof(F\K)$ (as the left class of a factorization system is always closed under pushouts). Note that maps in $\cof(F\K)$ are levelwise monomorphisms -- it follows that the map $g$ is at each spectrum level a pushout of a weak equivalence along a cofibration of simplicial sets, and is therefore a weak equivalence. Since $X'$ is an $\Omega$-spectrum, then so is $X'\cup_X Y$ as these are preserved by level equivalences. However, $X'\cup_X Y$ may not be levelwise Kan.

To remedy this, we post-compose the diagram above with a fibrant replacement 
\[\begin{tikzcd}
i\colon X'\cup_X Y\rar[hookrightarrow,"\sim"] & \overline{X'\cup_X Y}
\end{tikzcd}\] in the level model structure on $\mathrm{Sp}^\Sigma$. Letting $\J \subseteq \K$ denote the set of generating trivial cofibrations in $\mathrm{Sp}_l^\mathbb{N}$ and $F\J$ the corresponding generating trivial cofibrations in $\mathrm{Sp}_l^\Sigma$, we have $F\J \subseteq F\K$, so the map $i\in \cof(F\J)$ above is in fact also in $\cof(F\K)$; hence, so is the composite $ij$. The spectrum $\overline{X'\cup_X Y}$, which is levelwise Kan by construction, is also an $\Omega$-spectrum as $i$ is a level weak equivalence. Finally, the map $U(ig)$ is a level weak equivalence, and therefore a $\pi_*$-isomorphism. This proves condition (1).

To check condition (2), let $X$ be a symmetric spectrum such that $UX$ is fibrant. The cylinder object in $\mathrm{sSet}$
\[ \Delta[0]\sqcup \Delta[0]\hookrightarrow \Delta[1]\xrightarrow{\sim} \Delta[0] \]
yields the required path object in $\symsp$
\[ X\cong X^{\Delta[0]}\xrightarrow{\sim} X^{\Delta[1]} \twoheadrightarrow  X^{\Delta[0]\sqcup \Delta[0]}\cong X\times X.\]
To see this we observe that the cotensor $(-)^A$ commutes with the forgetful functor $U$, and that in $\seqsp$, the above two maps are a weak equivalence and a fibration, respectively, because the model structure on $\seqsp$ is simplicial \cite[Theorem 2.3]{BF}.

By \cref{thm:transfer}, this produces a model structure on $\mathrm{Sp}^\Sigma$ in which a spectrum $X$ is fibrant if and only if $UX$ is fibrant in $\seqsp$ (that is, $X$ is an $\Omega$-spectrum and levelwise Kan), and a map $f$ is a trivial fibration if and only if $Uf$ is a trivial fibration in $\seqsp$ (that is, $f$ is a level trivial fibration). Since fibrant objects and trivial fibrations uniquely determine the model structure (see \cite[Theorem 51.10]{Joyal}), we recover the stable model structure $\symsp$ of \cref{thm:stableMSsymsp} as desired. Moreover, \cref{thm:transfer} yields a description of the weak equivalences between fibrant objects: these are precisely the $\pi_*$-isomorphisms.
\end{proof}

\bibliographystyle{alpha}
\bibliography{references}

\end{document}